\documentclass[11pt,reqno]{amsart}

\usepackage{amssymb}
\usepackage{amsthm}
\usepackage{amsmath,amsfonts,amssymb}
\usepackage{mathrsfs}
\usepackage{graphics,color}
\usepackage{enumerate}
\usepackage{float}
\usepackage{verbatim}
\usepackage{url}

\newtheorem{theorem}{Theorem}[section]
\newtheorem{lemma}[theorem]{Lemma}
\newtheorem{proposition}[theorem]{Proposition}

\newtheorem{definition}[theorem]{Definition\rm}
\newtheorem{remark}{\it Remark\/}

\newcommand*{\R}{\ensuremath{\mathbb{R}}}

\newcommand*{\M}{\ensuremath{\gamma}}

\newcommand*{\D}{\ensuremath{\mathscr{D}}}
\renewcommand*{\O}{\ensuremath{\mathscr{U}}}

\renewcommand*{\S}{\ensuremath{\mathcal{S}}}

\newcommand*{\T}{\ensuremath{\mathbb{T}}}

\newcommand*{\dist}{\ensuremath{\mathrm{dist\,}}}
\renewcommand*{\div}{\ensuremath{\mathrm{div\,}}}

\newcommand*{\curl}{\ensuremath{\mathrm{curl\,}}}

\newcommand{\eps}{\varepsilon}

\begin{document}

\title{Relaxation of the Incompressible Porous Media Equation}

\author{L\'aszl\'o Sz\'ekelyhidi Jr.}
\address{Institut f\"ur Angewandte Mathematik, Universit\"at Bonn, Endenicher Allee 62, D-53115 Bonn}
\email{szekely@hcm.uni-bonn.de}

%\date{13.02.2011}
\thanks{The author was supported by the Hausdorff Center for Mathematics in Bonn. Part of the research took place in Madrid during the program "Calculus of Variations, Singular Integrals and Incompressible Flows". The author would like to thank D.~C{\'o}rdoba, D.~Faraco and F.~Gancedo for several very interesting conversations during this meeting, as well as acknowledge financial support and warm hospitality of UAM and ICMAT}

\begin{abstract}
It was shown recently by Cordoba, Faraco and Gancedo in \cite{CFG} that the 2D porous media equation admits weak solutions 
with compact support in time. The proof, based on the convex integration framework developed for the incompressible Euler equations in \cite{euler1}, uses ideas from the theory of laminates, in particular $T4$ configurations. In this note we calculate the explicit relaxation of IPM, thus avoiding $T4$ configurations. We then use this to construct weak solutions to the unstable interface problem (the Muskat problem), as a byproduct shedding new light on the gradient flow approach introduced 
by Otto in \cite{Otto1}.
\end{abstract}

\maketitle

\section{Introduction}

We consider the incompressible porous media equation (IPM) in a 2-dimensional bounded domain $\Omega\subset\R^2$. The flow is described in Eulerian coordinates by a velocity field $v(x,t)$ and a pressure $p(x,t)$ obeying the conservation of mass and the conservation of momentum in the form of Darcy's law:
\begin{align}
\partial_t\rho+\div(\rho v)&=0,\label{e:IPM1}\tag{IPM1}\\
\div v&=0,\label{e:IPM2}\tag{IPM2}\\
v+\nabla p&=-(0,\rho).\label{e:IPM3}\tag{IPM3}
\end{align}
Here we chose $x_1$ as the horizontal and $x_2$ as the vertical direction, with the gravitational constant normalized to be $1$. Equation \eqref{e:IPM2} amounts to prescibing the flow to be incompressible, and this is coupled with the assumption that there is no flux across the boundary $\partial\Omega$, i.e. $v\cdot \nu=0$ on $\partial\Omega$. 

As usual, a weak solution to the system \eqref{e:IPM1}-\eqref{e:IPM3} with initial data $\rho_0\in L^{\infty}(\Omega)$ is defined as
a pair $(\rho,v)$ with
$$
\rho\in L^{\infty}(\Omega\times (0,T)),\quad v\in L^{\infty}(0,T;L^2(\Omega)),
$$
such that for all $\phi\in C^{\infty}_c(\R^2\times\R)$ we have
\begin{eqnarray}
\int_0^T\int_{\Omega}\rho(\partial_t\phi+v\cdot \nabla \phi)\,dxdt+\int_{\Omega}\rho_0(x)\phi(x,0)\,dx&=&0,\label{e:wIPM1}\\
\int_0^T\int_{\Omega}v\cdot\nabla\phi\,dxdt&=&0\label{e:wIPM2},
\end{eqnarray}
and for all $\psi\in C^{\infty}_c(\Omega)$
\begin{equation}\label{e:wIPM3}
\int_{\Omega}(v+(0,\rho))\cdot\nabla^{\perp}\psi\,dx=0.
\end{equation}
We remark explicitly that \eqref{e:wIPM2} includes the no-flux boundary condition for $v$ whereas in \eqref{e:wIPM3} the pressure $p$ has been eliminated (observe also that there is no boundary condition on $p$). 

Recently, D.~C{\'o}rdoba, D.~Faraco and F.~Gancedo showed in \cite{CFG}, that on the 2-dimensional torus $\T^2$ the system \eqref{e:IPM1}-\eqref{e:IPM3} admits 
nontrivial weak solutions $(\rho,v)\in L^{\infty}(\T^2\times \R)$ with compact support in time. 
The proof is based on the framework developed in \cite{euler1} for the incompressible Euler equations, although 
there are several places where the authors had to modify the arguments. In technical terms, one of the crucial points in the general scheme of convex integration is to show that the relaxation with respect to the wave cone of a suitably defined 
constitutive set, the $\Lambda$-convex hull, contains the zero state in its interior. In \cite{CFG} it was observed, that due to a lack of symmetry induced by the direction of gravity, this condition seems to fail for IPM; instead, a systematic method for obtaining a suitably modified constitutive set was introduced, based on so-called degenerate $T4$ configurations. The advantage of the method used in \cite{CFG} is that it is rather robust, and can be used in situations where an explicit calculation of $K^\Lambda$ is out of reach due to the high complexity (see also \cite{KMS,MS}). Indeed, the same technique has been recently applied successfully to a large class of active scalar equations by R.~Shvydkoy \cite{Shvydkoy}.

On the other hand, there are certain advantages to obtaining an explicit formula for the $\Lambda$-convex hull $K^{\Lambda}$ rather than just showing that a fixed state is in the interior. The explicit formula allows one to identify ''compatible boundary and initial conditions'', for which the construction works. For the incompressible Euler equations, such initial conditions were called "wild initial data" in \cite{euler2}. For IPM a specific example of independent interest is related to the so-called Muskat problem, or 2-phase Hele-Shaw flow. Recall that \eqref{e:IPM1}-\eqref{e:IPM3} can be used to model the flow of two immiscible 
fluids of different densities in a porous medium, or, equivalently, in a Hele-Shaw cell \cite{SaffmanTaylor}. If initially the two fluids form a horizontal interface, with the heavier fluid on top, it is well known that the initial value problem is ill-posed in classical function spaces \cite{Wooding,SiegelCaflischetal,CG}.  Although some explicit solutions are known \cite{Howison}, there is no general existence theory, neither for the 
evolution problem for the interface, nor for weak solutions of IPM. As an application of the explicit formula for the $\Lambda$-convex hull, in this paper we construct weak solutions to IPM with initial data given by the unstable interface.

 After a normalization we may assume that the density $\rho(x,t)$, indicating whether the pores at time $t$ near location $x\in\Omega$ are filled with the lighter or the heavier fluid, takes the values $\pm 1$. Hence, from now on we will concentrate on constructing weak solutions in the sense \eqref{e:wIPM1}-\eqref{e:wIPM3} which satisfy
\begin{equation}\label{e:IPM4}\tag{IPM4}
|\rho(x,t)|=1\quad\textrm{ for a.e.}(x,t)\in \Omega\times(0,T).
\end{equation}
Note that the solutions constructed in \cite{CFG} also satisfy \eqref{e:IPM4}.

Our weak solutions are in good agreement and show interesting connections to predictions concerning the coarse-grained density and the growth of the mixing zone made in \cite{Otto1,Otto2}. In \cite{Otto1} F.~Otto introduced a relaxation approach for \eqref{e:IPM1}-\eqref{e:IPM3}, based on 
a gradient flow formulation of IPM and using ideas from mass transport. It was shown that under certain assumptions there exists a {\it unique} ''relaxed'' solution $\overline{\rho}$, representing a kind of coarse-grained density. Moreover, Otto showed in \cite[Remark 2.1]{Otto2} that, in general, the mixing zone (where the coarse-grained density $\overline{\rho}$ is strictly between $\pm 1$) grows linearly in time, with the possible exception of a small set of volume fraction 
$O(t^{-1/2})$. 
 
In Section \ref{s:micro} we show that the relaxed solution $\overline{\rho}$ from \cite{Otto1} is very closely related to the concept of subsolution, introduced in \cite{euler1} for the incompressible Euler equations, and in particular we construct weak solutions $\rho_k$ such that $\rho_k\overset{*}{\rightharpoonup} \overline{\rho}$. The interpretation is that $\overline{\rho}$ is the coarse-grained density obtained from $\rho_k$. It is interesting to note in this connection, that, although weak solutions are clearly not unique, there is a way to identify a selection criterion among subsolutions which leads to uniqueness. 

Moreover, we show that if the coarse-grained density is independent of the horizontal direction, the linear growth estimate of \cite{Otto2} is sharp, in the sense that there is no exceptional set. As a consequence, we can interpret the uniqueness result of Otto as selecting the subsolution with "maximal mixing". In this light it is of interest to note that the analogous criterion for the incompressible Euler equations would be "maximally dissipating" \cite{Dafermos,GigliOtto,euler2}.

The paper is organized as follows. In Section \ref{s:relaxations} we recall from \cite{CFG} how to reformulate \eqref{e:IPM1}-\eqref{e:IPM3} as a differential inclusion, and then we calculate explicitly the relaxation, more precisely the $\Lambda$-convex hull of the constitutive set. These computations form the main contribution of this paper. If one is only interested in weak solutions as defined in this introduction (where $v$ can be {\it unbounded}), the ''simpler'' computations in Propositions \ref{p:hull1} and \ref{p:perturbation1} suffice. However, for completeness we include the computations that are required for {\it bounded} velocity $v$ in Propositions \ref{p:hull2} and \ref{p:perturbation2}. We also show that the $\Lambda$-convex hulls obtained are stable under weak convergence, in the spirit of compensated compactness \cite{Tartar,DiPerna}.

Then, in Section \ref{s:constructions} we show how the explicit form of the $\Lambda$-convex hull can be used in conjunction with the Baire category method to obtain weak solutions. For the convenience of the reader we include the details of the Baire category method in the Appendix. Finally, in Section \ref{s:micro} we use the framework to construct weak solutions to the unstable interface problem.

%%%%%%%%%%%%%%%%%%%%%%%%%%%
\section{The relaxation of IPM}\label{s:relaxations}
%%%%%%%%%%%%%%%%%%%%%%%%%%%

We start by setting
\begin{equation}\label{e:u}
u:=2v+(0,\rho).
\end{equation}
Then \eqref{e:IPM1}-\eqref{e:IPM3} can be rewritten as
\begin{equation}\label{e:linear}
\begin{split}
\partial_t\rho+\div m&=0,\\
\div (u-(0,\rho))&=0,\\
\curl (u+(0,\rho))&=0,
\end{split}
\end{equation}
coupled with
\begin{equation}\label{e:constitutive}
\begin{split}
m&=\tfrac{1}{2}(\rho u-(0,1)),\\
|\rho|&=1
\end{split}
\end{equation}
for almost every $(x,t)$. As in \cite{CFG}, we interpret \eqref{e:linear}-\eqref{e:constitutive} 
as a differential inclusion: the {\it state variable} 
$$
(\rho,u,m):\Omega\times(0,T)\to\R\times\R^2\times\R^2
$$
is subject to a linear system of conservation laws \eqref{e:linear}, and should take values 
in a {\it constitutive set} determined by \eqref{e:constitutive}.

\bigskip

The reason for introducing $u$ is purely cosmetic, it makes \eqref{e:linear} symmetric. It also helps in simplifying some of the calculations below. Observe that \eqref{e:linear} can be easily written as a genuine differential inclusion. Namely, we see that
$$
\begin{pmatrix} \rho-u_2&u_1\\u_1&\rho+u_2\end{pmatrix}
$$
is $\curl$-free, hence (locally) equal to $\nabla^2\phi$ for some function $\phi$. Then $\rho=\tfrac{1}{2}\Delta \phi$ and
$\div (\tfrac{1}{2}\partial_t \nabla \phi+m)=0$, hence
$m=-\tfrac{1}{2}\partial_t\nabla \phi + \nabla^\perp\psi$ for some $\psi$. In particular we deduce that {\it for any $(x,t)$-periodic solution 
$$
(\rho,u,m):\T^3\to\R\times\R^2\times\R^2
$$ 
of \eqref{e:linear}, there exist periodic functions 
$$
\phi,\psi:\T^3\to\R
$$ 
such that}
\begin{equation}
\begin{split}
\rho&=\tfrac{1}{2}\Delta \phi,\\
u&=\tfrac{1}{2}\bigl(2\partial_{12}\phi,\partial_{22}\phi-\partial_{11}\phi\bigr),\\
m&=-\tfrac{1}{2}\partial_t\nabla \phi + \nabla^\perp\psi.
\end{split}
\end{equation}
The periodic solutions of \eqref{e:linear} are characterized by the {\it wave cone} $\Lambda$. 
This is easy to calculate, since it is clearly independent of $m$ and the other two equations 
have a div-curl structure. Hence
$$
\Lambda=\{(\rho,u,m):\,|\rho|^2=|u|^2\}.
$$
We conclude the existence of localized plane-waves, an observation which has already been made in \cite{CFG}:

\begin{lemma}\label{l:waves}
Let $\bar{z}=(\bar{\rho},\bar{u},\bar{m})\in\Lambda$. There exists a sequence 
$$
w_j=(\rho_j,u_j,m_j)\in C_c^{\infty}\left(B_1(0)\times (-1,1);\R\times\R^2\times\R^2\right)
$$
 solving \eqref{e:linear} such that
\begin{enumerate}
\item $\dist(w_j,[-\bar{z},\bar{z}])\to 0\textrm{ uniformly}$,
\item $w_j\rightharpoonup 0\textrm{ weakly in }L^2$,
\item $\int\int|w_j|^2\,dxdt\geq C|\bar{z}|^2$.
\end{enumerate}
\end{lemma}

Next, we calculate the $\Lambda$-convex hull of the set in state-space defined by \eqref{e:constitutive}. Recall (e.g. from \cite{KMS}, see also \cite{Pedregal,Kirchheim}) that for a closed set $K$ the $\Lambda$-convex hull is defined to be the largest closed set $K^{\Lambda}$ which cannot be separated from $K$: a state $z$ does not belong to $K^{\Lambda}$ if there exists a function $f$ which is $\Lambda$-convex in the sense that
$$
s\mapsto f(w_0+sw)\textrm{ is convex for all }w\in \Lambda
$$
such that $f\leq 0$ on $K$ and $f(z)>0$. It follows immediately from this definition that
$$
z_1,z_2\in K\textrm{ with }z_1-z_2\in \Lambda\,\Longrightarrow\,[z_1,z_2]\subset K^{\Lambda}.
$$
For further properties of $\Lambda$-convex hulls and functions, see \cite{Kirchheim}.

Now, let
$$
K=\{(\rho,u,m):\,|\rho|=1,\,m=\tfrac{1}{2}\rho u\}.
$$
Observe that \eqref{e:constitutive} is equivalent to
\begin{equation}\label{e:relaxation}
\bigl(\rho,u,m+\tfrac{1}{2}(0,1)\bigr)\in K\quad\textrm{Êa.e.}.
\end{equation}
We note also that in the absence of spatial boundaries, any solution
of \eqref{e:linear} with $(\rho,u,m)\in K$ for a.e.$(x,t)$ is also a solution of 
\eqref{e:linear}-\eqref{e:constitutive}. In particular, for solutions on the torus we can ignore the
additional constant vector $\tfrac{1}{2}(0,1)$.

\begin{lemma}
The function
\begin{equation}\label{e:function}
f(\rho,u,m):=\bigl|m-\frac{1}{2}\rho u\bigr|+\frac{1}{4}(\rho^2+|u|^2)
\end{equation}
is convex.
\end{lemma}

\begin{proof}
A short calculation and the triangle inequality shows that,
for any $(\rho,u,m)$ and $(\bar{\rho},\bar{u},\bar{m})$
$$
f(\rho+\bar{\rho}t,u+t\bar{u},m+t\bar{m})\geq f(\rho,u,m)+ct+\frac{t^2}{4}(\bar{\rho}^2+|\bar{u}|^2-2|\bar{\rho}\bar{u}|),
$$
where 
$$
c=\frac{1}{2}(\bar{\rho}+u\cdot\bar{u})-|\bar{m}-\tfrac{1}{2}(\bar{\rho}u+\rho\bar{u})|.
$$
The convexity of 
$$
t\mapsto f(\rho+\bar{\rho}t,u+t\bar{u},m+t\bar{m}),
$$
and hence of $f:\R\times\R^2\times\R^2\to\R$ follows.
\end{proof}

\bigskip

\begin{proposition}\label{p:hull1}
$$
K^{\Lambda}=\Biggr\{(\rho,u,m):\,\bigl|\rho\bigr|\leq 1,\,\bigl|m-\tfrac{1}{2}\rho u\bigr|\leq \tfrac{1}{2}(1-\rho^2)\Biggl\}.
$$
\end{proposition}

\begin{proof}
Let $|\rho_0|<1$ and $u_0\in\R^2$. Fix $e\in\R^2$ with $|e|=1$ and
define
\begin{eqnarray}
u_1=u_0+(1-\rho_0)e\quad&&\quad \rho_1=1\label{e:v1r1}\\
u_2=u_0-(1+\rho_0)e\quad&&\quad \rho_2=-1.\label{e:v2r2}
\end{eqnarray}
Then it is easy to see that  for any $m_1,m_2$
$$
(\rho_1-\rho_2,u_1-u_2,m_1-m_2)\in\Lambda,
$$
hence in particular by setting $m_i=\tfrac{1}{2}\rho_i u_i$ we obtain
that, with 
\begin{equation*}
\begin{split}
m_0:&=\tfrac{1}{2}(1+\rho_0)m_1+\tfrac{1}{2}(1-\rho_0)m_2\\
&=\tfrac{1}{2}\rho_0u_0+\tfrac{1}{2}(1-\rho_0^2)e,
\end{split}
\end{equation*}
the state $(\rho_0,u_0,m_0)$ is contained in $K^{\Lambda}$.
Thus any $(\rho,u,m)$ with
$$
|\rho|\leq 1,\quad \bigl|m-\tfrac{1}{2}\rho u\bigr|= \tfrac{1}{2}(1-\rho^2)
$$
is contained in $K^{\Lambda}$. Moreover, since for any $\bar{m}\in\R^2$ we have $(0,0,\bar{m})\in\Lambda$, we deduce that in fact any $(\rho,u,m)$ with 
$$
|\rho|\leq 1,\quad \bigl|m-\tfrac{1}{2}\rho u\bigr|\leq \tfrac{1}{2}(1-\rho^2)
$$
is contained in $K^{\Lambda}$.

To see that this set is in fact the whole hull, we note that
$$
K\subset\{z:\,g(z)\leq 1/2\},
$$
where 
\begin{equation}\label{e:function2}
g(\rho,u,m):=f(\rho,u,m)+\frac{1}{4}(\rho^2-|u|^2)
\end{equation}
and $f$ is the convex function in \eqref{e:function}. On the other hand observe
that $\rho^2-|u|^2$ is $\Lambda$-convex (in fact $\Lambda$-affine). Therefore $g$ is
 $\Lambda$-convex, and therefore
necessarily $K^{\Lambda}\subset\{z:\,f(z)\leq 1/2\}$.
This completes the proof.
\end{proof}

\begin{remark}
Recalling \eqref{e:relaxation}, we see that if $\mathcal{K}$ denotes the constitutive set given by \eqref{e:constitutive}, 
then $(0,0,0)\in \partial\mathcal{K}^{\Lambda}$. This was the observation made in \cite[Remark 4.1]{CFG}.
\end{remark}

\begin{proposition}
The set $K^{\Lambda}$ characterizes precisely the relaxation of the IPM system. In other words it is stable
under weak convergence in the following sense: let $(\rho_k,u_k,m_k)\in L^{\infty}\times L^2\times L^2$ be a sequence 
of weak solutions to the linear system \eqref{e:linear} such that 
$$
(\rho_k,u_k,m_k)\in K^{\Lambda}\textrm{ for a.e. }(x,t),
$$
and assume $(\rho_k,u_k,m_k)\rightharpoonup (\rho,u,m)$ in $L^2$. Then the limit also satisfies
$$
(\rho,u,m)\in K^{\Lambda}\textrm{ for a.e. }(x,t).
$$
Moreover, $K^{\Lambda}$ is the smallest set containing $K$ with this property.
\end{proposition}

\begin{proof}
The crucial information is that $\rho^2-|u|^2$ is continuous with respect to weak convergence, as a consequence of the div-curl lemma \cite{MuratTartar,Tartar,DiPerna}. Therefore,
if $g(\rho_k,u_k,m_k)\leq 0$ a.e. (where $g$ is the function from \eqref{e:function2}), then also $g(\rho,u,m)\leq 0$ a.e. 
\end{proof}

%%%%%%%%%%%%%%%%%%%%%%%%%%%%%%%%%%%%%%
Although working with the set $K$ would suffice to construct 
weak solutions to \eqref{e:IPM1}-\eqref{e:IPM4}, there is a slight drawback in that $K$
is not bounded. More specifically, the solutions constructed will have $v\in L^2$ but not in $L^{\infty}$. For the details
see the next section. To remedy this problem we next consider a compact subset of $K$. 
Here the calculation of the $\Lambda$-convex hull requires a bit more work.

Fix $\M>1$ and let
\begin{equation}\label{e:K}
K_\M=\{(\rho,u,m):\,|\rho|=1,\,m=\tfrac{1}{2}\rho u,\,|u|\leq \M\}.
\end{equation}

\begin{proposition}\label{p:hull2}
For the set $K_\M$ above with $\M>1$, $(K_\M)^{\Lambda}$ is given by the set of inequalities
\begin{align}
|\rho|&\leq 1,\label{e:1}\\
|u|^2&\leq \M^2-(1-\rho^2),\label{e:2}\\
|m-\tfrac{1}{2}\rho u|&\leq \tfrac{1}{2}(1-\rho^2),\label{e:3}\\
|m-\tfrac{1}{2}u|&\leq \tfrac{\M}{2}(1-\rho),\label{e:4}\\
|m+\tfrac{1}{2}u|&\leq \tfrac{\M}{2}(1+\rho).\label{e:5}
\end{align}
\end{proposition}

\begin{proof}
{\bf Step 1. }
Observe that \eqref{e:1},\eqref{e:4} and \eqref{e:5} define convex sets whereas
\eqref{e:2} and \eqref{e:3} define sublevel-sets of $\Lambda$-convex functions. Since states in $K_\M$ satisfy all 5 inequalities, it follows that $(K_\M)^{\Lambda}$ is certainly contained in the set defined by \eqref{e:1}-\eqref{e:5}. 

Conversely, let $U_\M$ be the set of all states $(\rho,u,m)$ with all inequalities \eqref{e:1}-\eqref{e:5} strict. We need to show
that $\overline{U_\M}\subset (K_\M)^{\Lambda}$. To this end it suffices to show that
\begin{equation}\label{e:extr1}
\begin{split}
\textrm{for all $z\in \partial U_\M\setminus K_\M$,}&\textrm{ there exists $\bar{z}\in \Lambda\setminus\{0\}$ such that}\\
z\pm\bar{z}&\in\partial U_\M.
\end{split}
\end{equation}
Indeed, from this it follows that $\textrm{extr }\overline{U_\M}\subset K_\M$, so that $\overline{U_\M}\subset (K_\M)^{\Lambda}$ (c.f. the Krein-Milman type theorem in the context of $\Lambda$-convexity \cite[Lemma 4.16]{Kirchheim}).
Thus, let $z=(\rho,u,m)\in \partial U_\M\setminus K_\M$. 

{\bf Step 2. }We first note that if \eqref{e:3} is an equality, i.e. 
$$
m=\frac{1}{2}\rho u+\frac{1}{2}(1-\rho^2)e
$$
for some $e\in\S^1$, then $z$ lies on a $\Lambda$-segment connecting $z_1,z_2\in K_\M$. More precisely,
if
\begin{equation*}
\begin{split}
z_1&=\left(1,u+(1-\rho)e,\tfrac{1}{2}(u+(1-\rho)e)\right),\\
z_2&=\left(-1,u-(1+\rho)e,\tfrac{1}{2}(u-(1+\rho)e)\right),
\end{split}
\end{equation*}
then $z=\tfrac{1}{2}(1+\rho)z_1+\tfrac{1}{2}(1-\rho)z_2$, as in Proposition \ref{p:hull1}. 

{\bf Step 3. }Next, by elementary computations we verify the following identity:
\begin{equation}\label{e:identity}
\begin{split}
&\left[\frac{(1-\rho^2)^2}{4}-|m-\frac{1}{2}\rho u|^2\right]+\frac{(1-\rho)^2}{4}\left[\M^2-|u|^2-(1-\rho^2)\right]=\\
=&\frac{1+\rho}{2}\left[\frac{\M^2}{4}(1-\rho)^2-|m-\frac{1}{2}u|^2\right]+\frac{1-\rho}{2}\left[\frac{\M^2}{4}(1+\rho)^2-|m+\frac{1}{2}u|^2\right]
\end{split}
\end{equation}
%%%%%%%%%%
\begin{comment}
Indeed, we have
\begin{equation*}
\begin{split}
|m-\tfrac{1}{2}\rho u|^2&=\left|(m-\tfrac{1}{2}u)+\tfrac{1-\rho}{2}u\right|^2\\
&=|m-\tfrac{1}{2}u|^2+(1-\rho)m\cdot u-\tfrac{1-\rho}{2}|u|^2+\tfrac{(1-\rho)^2}{4}|u|^2.
\end{split}
\end{equation*}
and similarly
\begin{equation*}
\begin{split}
|m-\tfrac{1}{2}\rho u|^2&=\left|(m+\tfrac{1}{2}u)-\tfrac{1+\rho}{2}u\right|^2\\
&=|m+\tfrac{1}{2}u|^2-(1+\rho)m\cdot u-\tfrac{1+\rho}{2}|u|^2+\tfrac{(1+\rho)^2}{4}|u|^2.
\end{split}
\end{equation*}
Combining the two we obtain
\begin{equation*}
|m-\tfrac{1}{2}\rho u|^2=\tfrac{1+\rho}{2}|m-\tfrac{1}{2}u|^2+\tfrac{1-\rho}{2}|m+\tfrac{1}{2}u|^2-\tfrac{(1-\rho^2)}{4}|u|^2.
\end{equation*}
\end{comment}
%%%%%%
Observe that the terms in square brackets are all non-negative in $\overline{U_\M}$. In particular, if \eqref{e:3} is strict, then either \eqref{e:4} or \eqref{e:5} has to be strict as well. On the other hand, if all three \eqref{e:3}-\eqref{e:5} are strict inequalities, then it is easy to verify that, with $\bar{z}=(1,e,0)$, 
$$
z+t\bar{z}\in\overline{U_\M}
$$
for all $t$ with $|t|$ sufficiently small, provided we choose $e\in S^1$ so that $\rho=u\cdot e$.

{\bf Step 4. }  
The remaining case is when \eqref{e:3} and one of \eqref{e:4} or \eqref{e:5} is a strict inequality and the other an equality. 
By symmetry, we may assume that \eqref{e:5} is strict, and 
$$
|m-\tfrac{1}{2}u|=\tfrac{\M}{2}(1-\rho).
$$
Choose $e\in S^1$ so that $\rho=u\cdot e$, set
$$
\bar{z}=\left(1,e,\tfrac{1}{2}e-\tfrac{1}{1-\rho}(m-\tfrac{1}{2}u)\right),
$$
and $z_t=(\rho_t,u_t,m_t):=z+t\bar{z}$. Then for all $t$
\begin{equation*}
\begin{split}
|u_t|^2&=\M^2-(1-\rho_t^2)\\
|m_t-\tfrac{1}{2}u_t|&=\tfrac{\M}{2}(1-\rho_t).
\end{split}
\end{equation*}
Moreover, by continuity, \eqref{e:3} and \eqref{e:5} will continue to hold for $|t|$ small enough. This implies \eqref{e:extr1} and
thus the proof is completed. 

\end{proof}

\begin{remark}
The condition $\M>1$ is sharp. Indeed, if $\M<1$, then it is not difficult to see (using standard technology on gradient Young measures) that 
approximate solutions to the corresponding inclusion are compact.  In fact, if $\Omega=\T^2$, by just looking at \eqref{e:IPM2}-\eqref{e:IPM3} 
we deduce easily that 
any weak solution $(\rho,u,m)$ of \eqref{e:linear}-\eqref{e:constitutive} with $\|u\|_{L^{\infty}}<1$ is constant (see for instance \cite{SMLecturenotes,Sverak}). 
\end{remark}

\begin{remark}
The computations in Proposition \ref{p:hull1} and \ref{p:hull2} do not depend on the vectors $u,m$ in state-space to be 2-dimensional. Therefore, the 
same formulae continue to hold for the relaxation of the IPM equation in $n$ space dimensions with any $n\geq 2$.
\end{remark}

%%%%%%%%%%%%%%%%%%%%%%%%
\section{Construction of Weak Solutions}\label{s:constructions}
%%%%%%%%%%%%%%%%%%%%%%%%%

There are several ways of constructing weak solutions to differential inclusions, depending
on the particular problems at hand. For the system \eqref{e:linear}-\eqref{e:constitutive} 
the relaxation is sufficiently large so that a relatively simple iteration procedure, involving single localized plane-waves, suffices. As always, the crucial ingredient is to show that there exists an open 
set $U$ where {\it states are stable only near }$K$ (c.f. \cite[Definition 3.16]{Kirchheim}). This is expressed by conditions (H1)-(H2) in the Appendix. The methodology of how to pass from this property to weak solutions via the Baire category theorem is well known \cite{Kirchheim,CFG,euler1,euler2}, but for the convenience of the reader we present the details in Theorem \ref{t:A1} in the Appendix. 

Recall the definition of $K$ and $K_\M$ from Section \ref{s:relaxations}, and 
let
\begin{equation}\label{e:U}
U=\textrm{int }K^{\Lambda},\quad U_\M=\textrm{int }(K_\M)^{\Lambda},
\end{equation}
so that 
\begin{equation}\label{e:ineq1}
U=\Biggr\{(\rho,u,m):\,\bigl|\rho\bigr|< 1,\,\bigl|m-\tfrac{1}{2}\rho u\bigr|< \tfrac{1}{2}(1-\rho^2)\Biggl\},
\end{equation}
and $U_\M$ is analogously given by the inequalities \eqref{e:1}-\eqref{e:5} all being strict.

\begin{proposition}\label{p:perturbation1}
There exists a constant $c>0$ with the following property.
For any $(\rho,u,m)\in U$ there exists $(\bar{\rho},\bar{u},\bar{m})\in\Lambda$ with $|\bar{\rho}|^2+|\bar{m}|^2=1$ such that
$$
(\rho,u,m)+t(\bar{\rho},\bar{u},\bar{m})\in U\quad\textrm{ for }|t|<c(1-\rho^2)
$$
\end{proposition}
\begin{proof}
From the explicit formulas \eqref{e:v1r1}-\eqref{e:v2r2} it follows that for any $(\rho,u,m)$ with $|\rho|<1$ and $|m-\tfrac{1}{2}\rho u|=\tfrac{1}{2}(1-\rho^2)$ there exists 
$(\bar{\rho},\bar{u},\bar{m})\in\Lambda$ with $\bar{\rho}=1$ such that
$$
(\rho,u,m)+t(\bar{\rho},\bar{u},\bar{m})\in K^{\Lambda}\quad\textrm{ for }|t|<c_0(1-\rho^2)
$$
for some $c_0>0$.
From this we deduce by continuity the claim (with $0<c<c_0$) in the case where 
$$
\tfrac{1}{4}(1-\rho^2)<|m-\tfrac{1}{2}\rho u|<\tfrac{1}{2}(1-\rho^2).
$$
In the remaining case, where $|m-\tfrac{1}{2}\rho u|<\tfrac{1}{4}(1-\rho^2)$, we can take
$(\bar{\rho},\bar{u},\bar{m})=(0,0,\bar{m})$ with $\bar{m}$ parallel to $m-\tfrac{1}{2}\rho u$.
\end{proof}

As a first application we obtain the following variant of \cite[Theorem 5.2]{CFG}:
%%%%%%%%%%%%%%%%%%%%
\begin{theorem}\label{t:1}
%%%%%%%%%%%%%%%%%%%%
There exist infinitely many periodic weak solutions to \eqref{e:IPM1}-\eqref{e:IPM3} 
with 
$$
\rho\in L^{\infty}(\T^2\times\R),\,v\in L^{\infty}(\R;L^2(\T^2)),
$$ 
such that 
$$
|\rho(x,t)|=\begin{cases} 1&\textrm{ a.e. }(x,t)\in \T^2\times (0,T),\\
0&\textrm{ for $t<0$ or $t>T$.}\end{cases}
$$ 
\end{theorem}

\begin{proof}
We construct solutions $(\rho,u,m)$ of \eqref{e:linear} such that
$$
(\rho,u,m)\in K\,\textrm{ a.e. }(x,t).
$$
For any such solution $|\rho|=1$ and $m=\tfrac{1}{2}\rho u$ almost everywhere by definition of $K$, and therefore, on the torus $\T^2$ satisfies
$$
\partial_t\rho+\tfrac{1}{2}\div(\rho u-(0,1)).
$$
Therefore, recalling that $v=\tfrac{1}{2}(u-(0,\rho))$ we deduce that $(\rho,v)$ is a weak solution of \eqref{e:IPM1}-\eqref{e:IPM3}. 

Next, define the space of {\it subsolutions} as follows. Let
\begin{equation*}
\begin{split}
\D&=\T^2\times \R,\\
\O&=\T^2\times (0,T),
\end{split}
\end{equation*}
and
$$
X_0=\Bigl\{z=(\rho,u,m)\in C^{\infty}_c(\O):\,\textrm{\eqref{e:linear} holds and }z(x,t)\in U\quad\forall\,(x,t)\in\O\Bigr\}.
$$
Note that $(0,0,0)\in U$, hence $X_0$ is nonempty.
Any $(\rho,u,m)\in X_0$ satisfies $|\rho|\leq 1$. Therefore, whenever $(\rho,u,m)\in X_0$, we have
$$
\|\rho\|_{L^{\infty}_tL^2_x(\D)}\leq 1,
$$
hence, using standard elliptic estimates and \eqref{e:linear},
$$
\|u\|_{L^{\infty}_tL^2_x(\D)}\leq C
$$
for some constant $C$. Together with the definition of $U$ this implies that $X_0$ is a bounded subset of $L^2_{x,t}(\D)$ and in particular 
$X_0$ satisfies condition (H3) in the Appendix.
Moreover, Lemma \ref{l:waves} and Proposition \ref{p:perturbation1} imply that (H1)-(H2) in the Appendix are satisfied, and consequently Theorem \ref{t:1} follows from applying Theorem \ref{t:A1} to $X_0$.

\end{proof}

\bigskip

\begin{proposition}\label{p:perturbation2}
Let $\M>2$. If $K_\M$ and $U_\M$ are as defined in \eqref{e:K} and \eqref{e:U}, then for any $z\in U_\M$ with $\dist(z,K_\M)\geq\alpha>0$ there exists $\bar{z}\in\Lambda$ with $|\bar{z}|=1$ such that
$$
z+t\bar{z}\in U_\M\textrm{ whenever }|t|<\beta,
$$
where $\beta>0$ depends only on $\alpha$ and $\M$. 
\end{proposition}

\begin{proof}
Since $\partial U_\M$ is compact and locally the graph of a Lipschitz function, it suffices to prove the following quantitative version of \eqref{e:extr1}:
\begin{equation}\label{e:extr2}
\begin{split}
\textrm{ for all }z_0\in \partial U_\M\setminus K&\textrm{ there exists $\eps>0,r>0$ so that}\\
\textrm{ for all }z\in B_{\eps}(z_0)&\cap\partial U_\M\textrm{ there exists }\bar{z}\in \Lambda\textrm{ with }\\
|\bar{z}|=r\,\textrm{ and }&\,z\pm \bar{z}\in \partial U_\M. 
\end{split}
\end{equation}
Let $z_0=(\rho_0,u_0,m_0)\in \partial U_\M\setminus K_\M$. 

If $|m_0-\tfrac{1}{2}\rho_0u_0|<\tfrac{1}{2}(1-\rho_0^2)$, then using the argument from Step 3 in the proof of Proposition \ref{p:hull2} we may assume that $|m_0+\tfrac{1}{2}u_0|<\tfrac{\M}{2}(1+\rho_0)$. But then the $\Lambda$-direction in Step 4 works uniformly for a whole neighbourhood of $(\rho_0,u_0,m_0)$.

Conversely, assume that $|m_0-\tfrac{1}{2}\rho_0u_0|=\tfrac{1}{2}(1-\rho_0^2)$, so that
$$
m_0=\tfrac{1}{2}\rho_0u_0+\tfrac{1}{2}(1-\rho_0^2)e_0\,\textrm{ with }|e_0|=1.
$$
We know from Step 2 of the proof of Proposition \ref{p:hull2} that there is a $\Lambda$-segment in $\partial U_\M$ of length $\min(1-\rho_0,1+\rho_0)$ centered at $(\rho_0,u_0,m_0)$, but now we need to show that the length of the segment can be chosen uniformly large for a whole neighbourhood. In other words we need to show that this $\Lambda$-segment is {\it stable}. 

To this end consider $z=(\rho,u,m)\in\partial U_\M$ nearby, so that in particular
$$
m=\tfrac{1}{2}\rho u+\tfrac{1}{2}(1-\rho^2)e+\tilde m\,\textrm{ with }|e|=1,
$$
with $|\rho-\rho_0|,|u-u_0|,|e-e_0|,|\tilde m|$ small. We may also assume that $|m-\tfrac{1}{2}\rho u|<\tfrac{1}{2}(1-\rho^2)$. Let 
$$
\bar{z}=\left(1,e,\tfrac{1}{2}e-\tfrac{1}{1-\rho}(m-\tfrac{1}{2}u)\right)
$$
as in Step 4 of the proof of Proposition \ref{p:hull2} and consider $z_t=(\rho_t,u_t,m_t)=m+t\bar{m}$. By continuity we can choose $e\in S^1$ in such a way that 
\begin{equation}\label{e:e}
\left(\rho,u,\tfrac{1}{2}\rho u+\tfrac{1}{2}(1-\rho^2)e\right)\in \partial U. 
\end{equation}
Since $\M>2$, such $e\in S^1$ exists for any $\rho,u$. Moreover, since $\M>2$, the circles defined by \eqref{e:4}-\eqref{e:5} are strictly larger than the one defined by \eqref{e:3}. 
It follows that either \eqref{e:e} is satisfied by all $e\in S^1$, or the circles defined by \eqref{e:4}-\eqref{e:5} intersect the circle defined by \eqref{e:3} transversely. In particular we may choose $e$ and consequently $\tilde m$ so that
\begin{equation}\label{e:angle}
-e\cdot \tilde m\geq \delta|\tilde m|,
\end{equation}
where $\delta>0$ depends only on $\M>2$. 

Now, since $e$ satisfies \eqref{e:e}, $(\rho_t,u_t)$ will satisfy \eqref{e:2} for all $t\in (-(1+\rho),1-\rho)$ (c.f. Step 2 in the proof of Proposition \ref{p:hull2}). Thus, using \eqref{e:identity} we see that $z_t$ remains inside $\partial U_\M$ as long as \eqref{e:3} remains a strict inequality. Now, a short calculation shows that
$$
m_t-\tfrac{1}{2}\rho_tu_t=\tfrac{1}{2}e(1-\rho_t^2)+\frac{1-\rho_t}{1-\rho}\tilde m.
$$
In particular $|m_t-\tfrac{1}{2}\rho_tu_t|=\tfrac{1}{2}(1-\rho_t^2)$ if either $\rho_t=1$ or $\rho_t$ satisfies
$$
(1+\rho_t)(1-\rho)e\cdot \tilde m=-|\tilde m|^2.
$$
Using \eqref{e:angle} we deduce that whenever $|\rho-\rho_0|,|u-u_0|,|e-e_0|,|\tilde m|$ is sufficiently small, $z_t\in \partial U_\M$ for all $t$ with $|t|<1/2\min(1-\rho_0,1+\rho_0)$.
This concludes the proof. 

\end{proof}

Using the same argument as in the proof of Theorem \ref{t:1} but replacing $K,U$ by $K_\M,U_\M$ with $\M>2$, 
we obtain the following strengthening (corresponding to Theorem 5.2 of \cite{CFG}):

\begin{theorem}\label{t:2}
There exist infinitely many periodic weak solutions to \eqref{e:IPM1}-\eqref{e:IPM3}
with $\rho,v\in L^{\infty}$, $|\rho(x,t)|=1$ for a.e. $(x,t)\in \T^2\times [0,T]$ and compact support in time. 
\end{theorem}

%%%%%%%%%%%%%%%%%%%%%%%%%%%%%%
\section{Evolution of Microstructure}\label{s:micro}
%%%%%%%%%%%%%%%%%%%%%%%%%%%%%%

In this section we show how to construct solutions to \eqref{e:IPM1}-\eqref{e:IPM4} in 
the spatial domain 
$$
\Omega:=(-1,1)\times(-1,1)
$$ 
which exhibit the type of 
mixing behaviour that is expected (\cite{CG,Otto1,Otto2}), when one starts with a horizontal interface with the heavier fluid on top. 
Thus, let 
\begin{equation}\label{e:unstable}
\rho_0(x)=\begin{cases} +1&x_2>0\\ -1& x_2<0\end{cases}
\end{equation}
and define subsolutions as follows.

Let
\begin{equation*}
\D=\Omega\times (0,T).
\end{equation*}
We consider triples $(\rho,v,m)\in L^{\infty}(\D)$ such that for all $\phi\in C_c^{\infty}(\overline{\Omega}\times [0,T))$
\begin{equation}
\int_0^T\int_{\Omega}\partial_t\phi\rho+\nabla\phi\cdot m\,dxdt+\int_{\Omega}\phi(x,0)\rho_0(x)\,dx=0,\label{e:a1}
\end{equation}
\begin{align}
\int_{\Omega}v\cdot\nabla\psi\,dx&=0\qquad\forall\psi\in C^{\infty}(\overline{\Omega}),\label{e:a2}\\
\int_{\Omega}(v+(0,\rho))\cdot\nabla^{\perp}\psi\,dx&=0\qquad\forall\psi\in C_c^{\infty}(\Omega),\label{e:a3}
\end{align}
and
\begin{equation}
|m-\rho v+\tfrac{1}{2}(0,1-\rho^2)|\leq \tfrac{1}{2}(1-\rho^2)\,\textrm{ in }\D.\label{e:a4}
\end{equation}
We assume that there exists an open subset $\O\subset\D$ such that 
\begin{align}
&|\rho|=1\textrm{ a.e. $\D\setminus\O$},\label{e:a5}\\
(\rho,v,m)&\textrm{ is continuous in }\O,\label{e:a6}\\
|m-\rho v+&\tfrac{1}{2}(0,1-\rho^2)|< \tfrac{1}{2}(1-\rho^2)\,\textrm{ in }\O.\label{e:a7}
\end{align}
\begin{definition}
Let us call any $(\rho,v,m)\in L^{\infty}(\D)$ satisfying \eqref{e:a1}-\eqref{e:a7} an {\it admissible subsolution} and the corresponding subset $\O\subset\D$ the {\it mixing zone}.
\end{definition}

Next, given an admissible subsolution $(\overline{\rho},\overline{v},\overline{m})$, let
\begin{equation*}
\begin{split}
X_0=\Bigl\{(\rho,u,m)\in L^{\infty}(\D):\,(\rho,v,m)&=(\overline{\rho},\overline{v},\overline{m})\textrm{ a.e. $\D\setminus \O$}\\
&\textrm{ and satisfies \eqref{e:a6}-\eqref{e:a7}}\textrm{ in $\O$}\Bigr\}.
\end{split}
\end{equation*}
As in the proof of Theorem \ref{t:1}, the space $X_0$ is bounded in $L^2(\D)$, hence weak $L^2$ convergence is metrizable on $X_0$. Let $X$ be the closure of $X_0$ with respect to the induced metric. 

\begin{theorem}\label{t:3}
There exists a residual set in $X$ consisting of weak solutions $(\rho,v)$ to \eqref{e:IPM1}-\eqref{e:IPM4} in $\Omega\times(0,T)$ with initial condition $\rho_0$.
\end{theorem}

\begin{proof}
Recall from \eqref{e:u} that $v=\tfrac{1}{2}(u-(0,\rho))$, so that \eqref{e:a7} is equivalent to
\begin{equation*}
\left(\rho,u,m+(0,\tfrac{1}{2})\right)\in U\quad\textrm{ for all }(x,t)\in\O,
\end{equation*}
where $U$ is defined in \eqref{e:U}. Therefore $X_0$ satisfies condition (H3) in the Appendix. Moreover, (H1) and (H2) are satisfied by Lemma \ref{l:waves} and
Proposition \ref{p:perturbation1}. Consequently Theorem \ref{t:3} follows from Theorem \ref{t:A1}.
\end{proof}

\begin{remark}\label{r:bounded}
If we replace \eqref{e:a4} and \eqref{e:a7} by the full set of inequalities \eqref{e:1}-\eqref{e:5} in Proposition \ref{p:hull2} for some $\M>2$,
and replace $U$ by $U_{\M}$ in the proof above, then Theorem \ref{t:3} remains true with the added information that the weak solutions satisfy $v\in L^{\infty}(\D)$. 
 \end{remark}
 
 \begin{remark}
 As a consequence of the residuality in $X$ we obtain, given an admissible subsolution $(\overline{\rho},\overline{v},\overline{m})$, the existence of a sequence of weak solutions $(\rho_k,v_k)$ of \eqref{e:IPM1}-\eqref{e:IPM4} such that
$$
\rho_k=\overline{\rho}\textrm{ a.e. in }\D\setminus \O\quad\textrm{ and }\quad\rho_k\overset{*}{\rightharpoonup} \overline{\rho}\textrm{ in }L^\infty(\O)
$$
as $k\to \infty$. In other words $\overline{\rho}$ represents a kind of coarse-grained density. This justifies calling $\O$ the mixing zone.
\end{remark}

\bigskip

To conclude this section we exhibit examples of (nontrivial) admissible subsolutions.
In particular we set
$$
\overline{v}\equiv 0,\quad \overline{m}=(0,\overline{m}_2),
$$
and assume that $(\overline{\rho},\overline{v},\overline{m})$ are just functions of $t$ and the ''height'' $x_2$, i.e.
$$
\overline{\rho}=\overline{\rho}(x_2,t),\quad \overline{m}_2=\overline{m}_2(x_2,t).
$$
Then \eqref{e:a2} and \eqref{e:a3} are automatically satisfied, whereas \eqref{e:a1} can be written as
\begin{align}
\partial_t\overline{\rho}+\partial_2\overline{m}_2&=0\quad\textrm{ in }(-1,1)\times(0,T),\label{e:sub1}\\
\overline{m}_2&=0\quad\textrm{ for }x_2=\pm 1,\label{e:sub2}\\
\overline{\rho}&=\begin{cases}-1&x_2<0\\ +1&x_2>0\end{cases}\quad\textrm{ for }t=0,\label{e:sub3}
\end{align}
interpreted in the weak formulation.
Furthermore, \eqref{e:a4} becomes
\begin{equation}\label{e:sub4}
-(1-\overline{\rho}^2)\leq \overline{m}_2\leq 0\quad\textrm{ in }(-1,1)\times(0,T).
\end{equation}
Note also in connection with Remark \ref{r:bounded} that in this case 
\eqref{e:2},\eqref{e:4} and \eqref{e:5} follow automatically from \eqref{e:sub4} provided $\gamma>3$.
 
An obvious way to construct admissible subsolutions is then to prescribe
$$
\overline{m}_2=-\alpha(1-\overline{\rho}^2)
$$
for $\alpha\in (0,1)$. The resulting equation from \eqref{e:sub1} is essentially the inviscid Burgers' equation, with the (unique) entropy solution
given by
\begin{equation}\label{e:subsol}
\overline{\rho}(x_2,t)=\begin{cases} -1& x_2<-2\alpha t,\\ \frac{x_2}{2\alpha t}&|x_2|<2\alpha t,\\ +1&x_2>2\alpha t\end{cases} 
\end{equation}
for times $t<1/(2\alpha)$. The corresponding mixing zone is
$$
\O=\{(x,t)\in\D:\,|x_2|<\alpha t\}.
$$
Since at $t=1/(2\alpha)$ we have $|\overline{\rho}|<1$ for all $x_2\in(-1,1)$, the functions $\overline{\rho},\overline{m}$ can easily 
be extended continuously to later times so that  \eqref{e:sub1}-\eqref{e:sub4} continue to hold.

It is interesting to note that in the borderline case $\alpha=1$ the subsolution $\overline{\rho}$ is precisely
the unique solution of the relaxation approach to the problem \eqref{e:IPM1}-\eqref{e:IPM4} introduced by F.~Otto in \cite{Otto1}. 
Thus, although weak solutions are clearly not unique, there seems to be a way to recover uniqueness 
at least for subsolutions. We plan to explore further this connection elsewhere.
Here we contend ourselves with showing that among all subsolutions, for which $\overline{\rho}$ is a function of $t$ and the vertical direction $x_2$ only, 
the case $\alpha=1$ in \eqref{e:subsol} corresponds to the one with ''maximal mixing''. This gives a new interpretation to the results in \cite{Otto1}.

\begin{proposition}
Let $(\overline{\rho},\overline{v},\overline{m})$ be an admissible subsolution to the unstable initial condition \eqref{e:unstable} such that $\overline{\rho}=\overline{\rho}(x_2,t)$. 
Then the mixing zone $\O$ is contained in $\{(x,t):|x_2|<2t\}$. 
\end{proposition}

\begin{proof}
We start by observing that since $\partial_{1}\overline{\rho}=0$, \eqref{e:a2}-\eqref{e:a3} imply that $\overline{v}=0$. 
Next, let
$$
\chi(s)=\begin{cases}0&s<0,\\s&s>0,\end{cases}
$$
and consider the test function $\phi(x,t)=\chi(x_2-2 t)$ in \eqref{e:a1} (it is easy to see by approximation that this is a valid test function). 
We obtain
$$
\int\int_{\{(x,t)\in\D:\,x_2>2t\}}\overline{m}_2-2\overline{\rho}\,dxdt+\int_{\{x\in\Omega:\,x_2>0\}}x_2\rho_0(x)dx=0.
$$ 
Therefore
$$
\int\int_{\{(x,t)\in\D:\,x_2>2t\}}(\overline{m}_2-2\overline{\rho}+2)\,dxdt=0.
$$
On the other hand \eqref{e:a4} implies
\begin{equation}\label{e:m2ineq}
\overline{m}_2\geq -(1-\overline{\rho}^2).
\end{equation}
We deduce that
$$
\overline{\rho}=1\textrm{ a.e. }(x,t)\in \{x_2>2t\}.
$$ 
Analogously we also find 
$$
\overline{\rho}=-1\textrm{ a.e. }(x,t)\in \{x_2< -2t\}.
$$ 
\end{proof}

\begin{comment}
\begin{remark}
Finally, we remark that for the stable initial condition 
$$
\rho_0(x)=\begin{cases} +1&x_2<0\\ -1& x_2>0\end{cases}
$$
instead of \eqref{e:sub3}, it is not difficult to see that the only admissible subsolution 
is the constant solution $\overline{\rho}(x,t)=\rho_0(x)$, for which there is no mixing zone.
\end{remark}
\end{comment}

%%%%%%%%%%%%%%%%%%%%%%%%%%%%%%%%%%
 \section{Appendix}
 %%%%%%%%%%%%%%%%%%%%%%%%%%%%%%%%%%
 
 We consider general systems in a domain $\D\subset\R^d$ of the form
 \begin{eqnarray}
 \sum_{i=1}^dA_i\partial_iz=0&&\textrm{ in }\D\label{e:LR}\\
 z(y)\in K&&\textrm{ a.e. }y\in\D\label{e:CR}
\end{eqnarray}
where
$$
z:\D\subset\R^d\to\R^N
$$
is the unknown state variable, $A_i$ are constant $m\times N$ matrices, and $K\subset\R^N$ is a closed 
set. We make the following assumptions.

\bigskip 

{\bf (H1)} The Wave Cone: There exists a closed cone $\Lambda\subset\R^N$ and a constant $C$ such that for all $\bar{z}\in\Lambda$ there exists a sequence $w_j\in C_c^{\infty}(B_1(0);\R^N)$ solving \eqref{e:LR} such that
\begin{enumerate}
\item $\dist(w_j,[-\bar{z},\bar{z}])\to 0$ uniformly,
\item $w_j\rightharpoonup 0$ weakly in $L^2$,
\item $\int |w_j|^2dy> C|\bar{z}|^2$.
\end{enumerate}

\bigskip

{\bf (H2)} The $\Lambda$-convex hull: There exists an open set $U\subset\R^N$ with $U\cap K=\emptyset$, and such that for all $z\in U$ with $\dist(z,K)\geq \alpha>0$ there exists $\bar{z}\in\Lambda\cap S^{N-1}$ such that
\begin{equation}
z+t\bar{z}\in U\textrm{ for all }|t|<\beta,
\end{equation}
where $\beta=\beta(\alpha)>0$. 

\bigskip

{\bf (H3)} Subsolutions: $X_0$ is a nonempty bounded subset of $L^2(\D)$ consisting of functions which are ''perturbable'' in an open subdomain $\O\subset\D$. This means that 
any $z\in X_0$ is continuous on $\O$ with
\begin{equation}\label{e:CRprime}
z(y)\in U\quad\textrm{ for }y\in\O,
\end{equation} 
and moreover, if $z\in X_0$ and $w\in C_c(\O)$ such that $w$ solves \eqref{e:LR} and $(z+w)(y)\in U$ for all $y\in\O$, then 
$z+w\in X_0$.

\bigskip

Finally, let $X$ be the closure of $X_0$ with respect to the weak $L^2$ topology. Since $X_0$ is bounded, the topology of weak $L^2$ convergence is metrizable on $X$, making it into a complete metric space. Denote its metric by $d_X(\cdot,\cdot)$.

\begin{theorem}\label{t:A1}
Assuming (H1)-(H2), the set 
$$
\{z\in X:\,z(y)\in K\textrm{ a.e. }y\in\O\}
$$
is residual in $X$. 
\end{theorem}

The proof relies on the following lemma, where, for notational convenience we set
$$
F(z):=\min(1,\dist(z,K)).
$$
\begin{lemma}\label{l:A1}
Let $z\in X_0$ with $\int_{\O}F(z(y))dy\geq\eps>0$. For all $\eta>0$ there exists $\tilde z\in X_0$ with $d_X(z,\tilde z)<\eta$ and $\int_{\O}|z-\tilde z|^2dy\geq\delta$, where $\delta=\delta(\eps)>0$.
\end{lemma}

\begin{proof}[Proof of Lemma \ref{l:A1}]
Since $z\in X_0$ is continuous and $z(y)\notin K$ on $\O$, for any $y_0\in\O$ there exists $r_0=r_0(y_0)>0$ such that
\begin{equation}\label{e:A1}
\frac{1}{2}F(z(y_0))\leq F(z(y))\leq 2\,F(z(y_0))
\end{equation}
for all $y\in B_{r_0}(y_0)\subset\O$.
Then, by a simple domain exhaustion argument, we find disjoint balls 
$B_i:=B_{r_i}(y_i)\subset\O$ for $i=1\dots k$ such that
\begin{eqnarray}
m:=\biggl|\bigcup_{i=1}^kB_i\biggr|&\geq& \frac{1}{2}|\O|,\label{e:A2}\\
\sum_i\int_{B_i}F(z(y))\,dy&\geq& \frac{1}{2}\int_{\O}F(z(y))\,dy\label{e:A3}.
\end{eqnarray}

Next, observe that (H2) implies the existence of a continuous function $\phi:[0,1]\to[0,1]$ such that
$\phi(0)=0$, $\phi(t)>0$ for $t>0$ and such that
for any $z\in U$ there exists $\bar{z}\in \Lambda\cap S^{N-1}$ with
\begin{equation}\label{e:A4}
z+t\bar{z}\in U\quad\textrm{ whenever }|t|<\phi\bigl(F(z(y))\bigr).
\end{equation}
By considering its convexification if necessary, we may assume without loss of generality that $\phi$ is convex and monotone increasing.
Using \eqref{e:A1}, \eqref{e:A3} and the convexity of $\phi$ we obtain
\begin{equation*}
\begin{split}
\phi\biggl(\frac{1}{4|\O|}\int_{\O} F(z(y))dy\biggr)&\leq  \phi\biggl(\sum_i F(z(y_i))\frac{|B_i|}{m}\biggr)\\
&\leq \sum_i \phi\bigl(F(z(y_i))\bigr)\frac{|B_i|}{m}.
\end{split}
\end{equation*}
Moreover, using \eqref{e:A4}, (H1) and a simple rescaling, there exist $\tilde z_i\in C_c^{\infty}(B_i;\R^N)$ such that
\begin{enumerate}
\item $z(y)+\tilde z_i(y)\in U$ for all $y$, 
\item $d(\tilde z_i,0)<\eta/k$,
\item $\int_{B_i}|\tilde z_i|^2dy> C\phi\bigl(F(z(y_i))\bigr)|B_i|$,
\end{enumerate}
where $C$ is the constant in (H1).
Therefore,
$$
w:=\sum_i\tilde z_i\in C_c(\O)    
$$  
and $z(y)+w(y)\in U$ for all $y\in\O$, hence $\tilde z:=z+w\in X_0$ by (H3).  
Moreover, $\tilde z\in X_0$ satisfies
\begin{eqnarray*}
d_X(\tilde z,z)&\leq& \sum_id_X(\tilde z_i,0)<\eta\\
\int_{\O}|z-\tilde z|^2dy&\geq& C\sum_i\phi\bigl(F(z(y_i))\bigr)|B_i|\\
&\geq& \frac{C}{2|\O|}\phi\biggl(\frac{1}{4|\O|}\int_{\O}F(z(y))dy\biggr).
  \end{eqnarray*}
 This concludes the proof. 
 \end{proof}
 
 \begin{proof}[Proof of Theorem \ref{t:A1}]
 First of all the functional $I(z)=\int_{\O}|z(y)|^2\,dy$ is a Baire-1 function on $X$. Indeed, observe that
  $$
 I_j(z):=\int_{\{y\in\O:\dist(y,\partial\O)>1/j\}}|z*\rho_j(y)|^2dy,
 $$
 where $\rho_j\in C_c^{\infty}(B_{1/j}(0))$ is the usual mollifier sequence, is continuous as a map $X\to\R$, and moreover
 $$
 I_j(z)\to I(z)\textrm{ as }j\to\infty.
 $$
 Therefore, by the Baire category theorem the set 
 $$
 Y:=\{z\in X:\,I\textrm{ is continuous at }z\}
 $$
 is residual in $X$. We claim that $z\in Y$ implies $\int_{\O}F(z(y))dy=0$. 
 
 If not, let $\eps:=\int_{\O}F(z(y))dy>0$ for some $z\in Y$, and let $z_j\in X_0$ be a sequence such that
 $z_j\overset{d}{\to} z$ in $X$. Since $I$ is continuous at $z$, it follows that $z_j\to z$ strongly in $L^2(\O)$, and in particular we may assume that $\int_{\O}F(z_j(y))dy>\eps/2$. 
 
 Then, by applying Lemma \ref{l:A1} to each $z_j$, we obtain a new sequence $\tilde z_j\in X_0$ such that
 $\tilde z_j\overset{d}{\to} z$ in $X$ (and hence strongly in $L^2$), but $\int_{\O}|z_j-\tilde z_j|^2dy\geq \delta>0$, where $\delta$ only depends on $\eps$. This contradicts the strong convergence of both sequences to $z$. 
 
 \end{proof}

%\bibliographystyle{acm}
%\bibliography{ipm-literature}

\end{document}